\documentclass{amsart}
\usepackage{amssymb}
\usepackage{bbm}
\usepackage{amsfonts}
\usepackage{amscd}
\usepackage{mathrsfs}
\usepackage{amsmath}
\usepackage{bm}
\usepackage{enumitem}
\newtheorem{theorem}{Theorem}[section]
\newtheorem{lemma}[theorem]{Lemma}
\newtheorem{p}{Proposition}[section]
\newtheorem{corollary}{Corollary}[section]
\theoremstyle{definition}
\newtheorem{definition}[theorem]{Definition}
\newtheorem{example}[theorem]{Example}

\theoremstyle{remark}
\newtheorem{remark}[theorem]{Remark}
\numberwithin{equation}{section}

\newcommand{\blankbox}[2]

\begin{document}
	\title{ A class of uniformly bounded simple $\mathbb{Z}$-graded Lie conformal algebras}
	
	\author{Maosen Xu}
	\address{School of Mathematics Information, Shaoxing University, Shaoxing, 312000, P.R.China} \email{390596169@qq.com}

	
	\begin{abstract}In this paper, we classify the  following simple  $\mathbb{Z}$-graded Lie conformal algebras $\mathcal{L}=\bigoplus_{i\in \mathbb{Z}}\mathcal{L}_i$ such that (1)$rank\mathcal{L}_i\leq 1$, (2)$\mathcal{L}_0$ is the Virasoro Lie conformal algebra.
	\end{abstract}
	\maketitle

	\section{introduction}
	The notion of Lie conformal algebra was introduced to provide an axiomatic description of
	properties of the operator product expansion in conformal field theory in \cite{K}. Many other fields are closely related to Lie conformal algebras such as vertex algebras and linearly compact Lie algebras. Furthermore, Lie conformal algebra is interpreted as the Lie algebra over a pseudo-tensor category in \cite{BDK}.
	The theory of finite Lie conformal algebras was studied systematically in \cite{DK}.   
	In particular,  finite simple Lie conformal algebras were classified in \cite{DK}, which turns out to be either the Virasoro Lie conformal algebra or the current Lie conformal algebra of a finite-dimensional simple Lie algebra (up to isomorphism).  The description of finite irreducible modules for finite simple Lie conformal algebras was given in \cite{CK}.  Besides, the structure of infinite simple Lie conformal algebra $gc_N$ and its simple subalgebras were extensively studied (see [BKL1], [BKL2], [SY]).   
	
	Compared with the finite case, the classification of infinite simple Lie conformal algebras of finite growth is a challenging problem as stated in  \cite{K1}.     Even the uniformly bounded case has not yet been completed. Fortunately, some methods still exist to obtain some simple $\mathbb{Z}$-graded Lie conformal algebras. Y. Su and X. Yue obtained the Block type Lie conformal algebra from $gc_1$ in \cite{SY}.  X. Xu established the relation between the quadratic Lie conformal algebras and Gel'fand-Dorman algebras in \cite{X}. Further, Y. Hong and Z. Wu studied the simplicity between these two kinds of algebras in \cite{HW}, from which two simple  
	$\mathbb{Z}$-graded Lie conformal algebras  $CL_1(s)$ and  $CL_2(b,s)$ can be found.

	Inspired by the structure of $CL_1(s)$, $CL_2(b,s)$, this paper aims to classify the $\mathbb{Z}$-graded simple Lie conformal algebras $\mathcal{L}=\bigoplus_{i\in \mathbb{Z}} \mathcal{L}_i$ such that :  \begin{enumerate}
		\item[(C1)] $rank\mathcal{L}_i \leq 1$; 
		\item[(C2)] $\mathcal{L}_0$ is the Virasoro Lie conformal algebra.                 
	\end{enumerate}

	The paper is organized as follows.
	
	Section 2 recalls the basic notions and propositions of Lie conformal algebras, Novikov algebras, and Gel'fand-Dorfman algebras.
	
	In Section 3,   we first obtain some properties of the structure constants of $\mathcal{L}$. Secondly, we characterize the structure of subalgebras of $\mathcal{L}$ generated by $\mathcal{L}_{k}$ and $\mathcal{L}_{-k}$ when $[{\mathcal{L}_k}_\lambda \mathcal{L}_{-k}]\neq 0$.  Finally,  the main classification results are given as follows.
	\begin{theorem}{\label{mt}}
		Suppose that $\mathcal{L}$ is a  $\mathbb{Z}$-graded simple Lie conformal algebra satisfying $(C1)$ and $(C2)$.  Then $\mathcal{L}$ must be isomorphic to:
		\begin{enumerate}
			\item[$\bullet$] $Vir$.
			\item[$\bullet$]  $\mathcal{V}(s)=\bigoplus_{i \in \mathbb{Z}}\mathbb{C}[\partial]L_i, s\in \mathbb{C}:  \text{for}\ \  i,j\in \mathbb{Z},\ \ [{L_i}_\lambda L_j]=(\partial+2\lambda+s(i-j))L_{i+j}$.
			\item[$\bullet$] $CL_1(s)=\bigoplus_{i \in \mathbb{Z}_{\geq -1}}\mathbb{C}[\partial]L_i,s\in \mathbb{C}:$ for $i,j \geq -1,$ \[ [{L_i}_\lambda L_j]=((i+1)\partial+(i+j+2)\lambda+s(j-i))L_{i+j}.\]
			\item[$\bullet$]  $CL_2(b,s)=\bigoplus_{i \in \mathbb{Z}}\mathbb{C}[\partial]L_i,2b\not\in \mathbb{Z},s\in \mathbb{C}: \text{for}\ \  i,j\in \mathbb{Z}$,\[[{L_i}_\lambda L_j]=((i+b)\partial+(i+j+2b)\lambda+s(i-j))L_{i+j}. \]
			
			\item[$\bullet$]   $SCL_2(b,s)=\bigoplus_{i \in \mathbb{Z}}\mathbb{C}[\partial]L_i,0\neq 2b\in \mathbb{Z}, s\in \mathbb{C}: $\[ [{L_0}_\lambda L_{-2b}]=b(\partial+\lambda+2s)L_{-2b},\]
			\[  [{L_{-2b}}_\lambda L_{-2b}]=-b(-\lambda+2s)(\partial+\lambda+2s)(\partial+2\lambda)L_{-4b},\]
			\[ [{L_i}_\lambda  L_{-2b}]=(\partial+\lambda+2s)((i+b)\partial+i\lambda+s(i+2b))L_{i-2b},\   i\neq -2b,\]
			\[[{L_i}_\lambda L_j]=((i+b)\partial+(i+j+2b)\lambda+s(i-j))L_{i+j},\  \text{if} \   i+j \neq -2b,\]
			\[ [{L_i}_\lambda L_j]=\frac{i-j}{2}L_{i+j},\  \text{if}\  i+j=-2b\  \text{and} \   i,j\neq 0. \]
		\end{enumerate}
	\end{theorem}
	
	Through this paper, denote $\mathbb{C}$ and $\mathbb{Z}_{\geq -1}$ the sets of all complex numbers and all integers greater than $-2$ respectively. $\mathbb{R}$ is the set of real numbers and $\mathbb{C}^*$ is the set of non-zero complex numbers.  $\mathbb{N}$ is the set of natural numbers. In addition, all vector spaces and tensor products are over $\mathbb{C}$. For any vector space $V$,
	we use $V[\lambda]$ to denote the set of polynomials of $\lambda$ with coefficients in $V$.  Let $f(\partial, \lambda)\in \mathbb{C}[\lambda,\partial]$. We denote the total degree of $f(\partial, \lambda)$ by $\text{deg}\   f(\partial, \lambda)$. For any complex number $a=x+yi$ where $x$, $y\in \mathbb{R}$, we denote
	$Re(a)=x$.
	
	\section{Preliminary}
	In this section, we recall some basic definitions and results of Lie conformal algebras and Gel'fand-Dorfman algebra. These facts can be found in \cite{DK} and \cite{HW}.
	
	\begin{definition}\label{d2.1}  A {\it Lie conformal algebra} $\mathcal{A}$ is a $\mathbb{C}[\partial]$-module together  with a $\mathbb{C}$-linear map (call $\lambda$-bracket)  $[\cdot_\lambda \cdot]:  \mathcal{A} \otimes \mathcal{A} \rightarrow \mathcal{A}[\lambda]$, $a\otimes b \mapsto [a_\lambda b]$, satisfying the following axioms
		\begin{equation*}
			\begin{aligned}
				&[\partial a_\lambda b]=-\lambda[a_\lambda b], \   [a_\lambda \partial b]=(\partial+\lambda)[a_\lambda b],\ \ \text{(conformal\   sequilinearity)},\\
				&[a_\lambda b]=-[b_{-\lambda-\partial}a]\ \ \text{(skew-symmetry)},\\
				&[a_\lambda [b_\mu c]]=[[a_\lambda b]_{\lambda+\mu}c]+[b_\mu[a_\lambda c]]\ \ \text{(Jacobi-identity)},
			\end{aligned}
		\end{equation*}
		for all $a$, $b$, $c\in \mathcal{A}$.
	\end{definition}
	Using conformal sequilinearity, we can define a Lie conformal algebra by giving the $\lambda$-brackets on its generators over $\mathbb{C}[\partial]$. In addition, the \emph{rank} of a Lie conformal algebra $\mathcal{A}$ is  its rank as a $\mathbb{C}[\partial]$-module, i.e.
	\[rank{\mathcal{A}}=\text{dim}_{\mathbb{C}(\partial)}(\mathbb{C}(\partial)\otimes_{\mathbb{C}[\partial]}\mathcal{A}).\]  We say a Lie conformal algebra {\it finite} if it is finitely generated as a $\mathbb{C}[\partial]$-module. 
	
	Suppose that $\mathcal{A}$ is a Lie conformal algebra.  For any $a,b \in \mathcal{A}$, we write
	\begin{eqnarray}
		\begin{array}{ll}
			[a_\lambda b]=\sum_{j\in \mathbb{Z}^+}(a_{(j)}b)\frac{\lambda^j}{j!}.
		\end{array}
	\end{eqnarray}
	For every $j\in \mathbb{Z}_+$, we have the $\mathbb{C}$-linear map: $\mathcal{A}\otimes  \mathcal{A} \rightarrow  \mathcal{A}$, $a\otimes b \mapsto a_{(j)}b$, which is called {\it $j$-th product}. If a Lie conformal algebra $\mathcal{A}$ is free as a $\mathbb{C}$-module, then there exists some vector space $V$ such that $\mathcal{A}=\mathbb{C}[\partial]\otimes V$. 
	The 0-th product  provides a Lie algebra structure over $V$, which is denoted by $L^0(\mathcal{A})$.
	\begin{example}
		The Virasoro Lie conformal algebra  $Vir=\mathbb{C}[\partial]L$ is a free  $\mathbb{C}[\partial]$-module of rank one, whose $\lambda$-brackets are determined by  $[L_\lambda L]=(\partial+2\lambda)L$. Furthermore, it is well known that $Vir$ is a simple Lie conformal algebra.
	\end{example}
	
	\begin{example}For a Lie algebra $\mathfrak{g}$, the current Lie conformal algebra $\text{Cur}\mathfrak{g}$ is a free $\mathbb{C}[\partial]$-module $\mathbb{C}[\partial] \otimes \mathfrak{g}$  equipped with  $\lambda$-brackets:
		\[ [x_\lambda y]=[x,y],\ \ \text{for}\ \  x,y \in \mathfrak{g}. \]
	\end{example}
	
	\begin{definition} For a Lie conformal algebra $\mathcal{A}$, a {\it conformal $\mathcal{A}$-module} $M$  is a $\mathbb{C}[\partial]$-module with a $\mathbb{C}$-linear map  $\mathcal{A} \otimes M \rightarrow M[\lambda]$, $a\otimes m \mapsto a_\lambda m$, satisfying the following axioms:
		\begin{equation*}
			\begin{aligned}
				&\partial a_\lambda m=-\lambda(a_\lambda m), \ \   a_\lambda \partial m=(\partial+\lambda)a_\lambda m,\\
				&a_\lambda (b_\mu m)=[a_\lambda b]_{\lambda+\mu}m+b_\mu(a_\lambda m),
			\end{aligned}
		\end{equation*}
		for all $a$, $b\in \mathcal{A}$ and $m\in M$.
	\end{definition}
	An $\mathcal{A}$-module $M$ is said to be \emph{trivial} if $a_\lambda m=0$ for any $a\in\mathcal{A}$ and $m\in M$.
	Suppose that $M$ is a finite  $\mathcal{A}$-module. Let \[ \text{Tor}M:=\{m\in M \mid f(\partial)m=0\  \text{for}\  \text{some nonzero}\  f(\partial) \in \mathbb{C}[\partial]\}.\] Then by  \cite[Lemma 8.2]{DK}, $\text{Tor}(M)$ is  trivial as an  $\mathcal{A}$-module.
	In particular, any simple Lie conformal algebra must be free as a $\mathbb{C}[\partial]$-module. 
	
	\begin{definition} We say a Lie conformal algebra $\mathcal{G}$ is \emph{$\mathbb{Z}$-graded} if $\mathcal{G}=\oplus_{i\in \mathbb{Z}}\mathcal{G}_i$, where each  $\mathcal{G}_i$ is a $\mathbb{C}[\partial]$-module and  $[{\mathcal{G}_i}_\lambda \mathcal{G}_j]\subset \mathcal{G}_{i+j}[\lambda]$ for any $i,j \in \mathbb{Z}$. In addition, we say a  $\mathbb{Z}$-graded Lie conformal algebra being   \emph{graded simple} if it has no proper graded ideals.

		Suppose that  $M$ is a $\mathcal{G}$-module. We say $M$ is  $\mathbb{Z}$-graded  if $M=\oplus_{i\in \mathbb{Z}}\mathcal{M}_i$, where each  $\mathcal{M}_i$ is a $\mathbb{C}[\partial]$-module and  ${\mathcal{G}_i}_\lambda \mathcal{M}_j\subset \mathcal{M}_{i+j}[\lambda]$ for any $i,j \in \mathbb{Z}$. Furthermore, if each $M_i$ is free of rank one as $\mathbb{C}[\partial]$-modules, then  $M$ is called as a \emph{free intermediate series module} of $\mathcal{G}$.
	\end{definition}
	
	\begin{definition}A \emph{Novikov algebra} $V$ is a vector space over $\mathbb{C}$ with a bilinear product $\circ: V \times V \to V$ satisfying
		\[ (a \circ b) \circ c-a \circ (b \circ c) = (b \circ a) \circ c-
		b \circ (a \circ c),\]
		\[ (a\circ b) \circ c = (a \circ c) \circ b,\ \ \   a,b,c\in V \] 
		A subspace $I$ is called an \emph{ideal} if $V\circ I \subset I$ and $I\circ V\subset I$. We say  $V$ is \emph{simple} if $V$ has no proper non-zero ideals. 
	\end{definition}
	\begin{p}{\label{OP}}\cite[Theorem 1.3]{O} Suppose that $A$ is an infinite dimensional simple Novikov algebra containing an element $e$ such that $e\circ e=be$ for some $b\in\mathbb{C}^*$ and $A$ is the direct sum of eigenspace of $e$. Then $A$ must be isomorphic to one of follows:
		\begin{enumerate} 
			\item $A_1=\oplus_{i\geq -1}\mathbb{C}L_i$, and for each $i,j\geq -1$,
			\[ L_i\circ L_j =(j + 1)L_{i+j}.\] 
			\item $A_2=\oplus_{\alpha \in \Delta}\mathbb{C}L_\alpha$, where $\Delta$ is an additive subgroup of $\mathbb{C}$. For each $\alpha,\beta\in \Delta$, the $\circ$ product is given as follows
			\[ L_\alpha \circ L_\beta=(\beta+b)L_{\alpha+\beta}.\]
			\item $A_3=\oplus_{(\alpha,n) \in \Delta\times \mathbb{N}}\mathbb{C}L_{(\alpha,n)}$, where $\Delta$ is an additive subgroup of $\mathbb{C}$. Then
			\[ L_{\alpha,i}\circ L_{\beta,j} = (\beta+b)L_{\alpha+\beta,i+j}+jL_{\alpha+\beta,i+j-1}, \]
			for each $\alpha, \beta \in  \Delta, i,j\in  \mathbb{N}$.
		\end{enumerate}
	\end{p}
	\begin{remark}
		A list of serious accuracies in \cite{O} was listed in \cite{X}. Notice that results in \cite{O} are all correct if $\mathbb{C}$ is chosen as the base field, as we do in this paper. 
	\end{remark}
	\begin{definition} Suppose that $V$ is a vector space. We say a Lie conformal algebra $R=\mathbb{C}[\partial]V$ is  \emph{quadratic}  if  for any $a,b \in V$,
		\[ [a_\lambda b]=u\partial+v\lambda+w,\]
		for some $u,v,w \in V$.
	\end{definition}
	
	\begin{definition} A \emph{Gel’fand-Dorfman algebra} is a triple $(V, \circ, [\cdot,\cdot])$, where $V$ is a vector space over $\mathbb{C}$, $(V,\circ)$ forms a Novikov algebra and $(V,[\cdot,\cdot])$ forms a Lie algebra, together with 
		following compatibility conditions:
		\[ [a \circ b, c]-[a \circ c, b]+[a, b] \circ c-[a, c] \circ b-a \circ [b, c]=0, \]
		for $a$, $b$ and $c \in V$.
	\end{definition}
	The Gel'fand-Dorfman algebra structure over $A_1$ and $A_2$ has been determined as follows.
	\begin{p}\label{XP}(\cite[Theorem 3.1]{X})
		Any Gel'fand-Dorfman algebra structure  over the  Novikov algebra  $A_1$ is
		given as follows: for all  $i,j\geq -1$,
		\[ [L_i,L_j] = s(i-j)L_{i+j},\    \text{for}\  \text{some} \ s \in  \mathbb{C}.\]
		Any Gel'fand-Dorfman algebra structure  over Novikov algebra the  $A_2$ is
		given as follows:  for  all $i,j \in \mathbb{Z}$,
		\[ [L_i,L_j]=s(i-j)L_{i+j}, \ \text{for}\  \text{some} \ s \in  \mathbb{C}.\]
	\end{p}
	\begin{p} \cite{X}. Suppose that $V$ is a vector space. Then $R=\mathbb{C}[\partial]V$ is a quadratic Lie conformal algebra if and only if the $\lambda$-bracket of $R$ is given as follows
		\[ [a_\lambda b]=\partial(b\circ a)+[b,a]+\lambda(a\ast b) , a, b\in V.\]
		where $a\ast b=b\circ a +a\circ b$ and $(V,\circ,[\cdot,\cdot])$ is a Gel'fand-Dorfman algebra. Therefore, $R$ is called the quadratic Lie conformal algebra corresponding to the Gel'fand-Dorfman algebra  $(V,\circ,[\cdot,\cdot])$. Conversely, we use $R_N$ and $R_{GD}$ to denote the corresponding Novikov algebra and Gel'fand-Dorfman algebra of $R$.
	\end{p}
	We can get  two $\mathbb{Z}$-graded quadratic Lie conformal algebras from the Gel'fand-Dorfman algebra $A_1$,$A_2$ in Proposition \ref{XP}  as follows: for some $b,s\in \mathbb{C}$,
	\[ CL_1(s)=\bigoplus_{i \in \mathbb{Z}_{\geq -1}}\mathbb{C}[\partial]L_i: [{L_i}_\lambda L_j]=((i+1)\partial+(i+j+2)\lambda+s(j-i))L_{i+j},\]
	\[ CL_2(b,s)=\bigoplus_{i \in \mathbb{Z}}\mathbb{C}[\partial]L_i:[{L_i}_\lambda L_j]=((i+b)\partial+(i+j+2b)\lambda+s(i-j))L_{i+j}\]
	As stated in \cite{HW},  $CL_2(b,s)$ is not simple when $2b \in \mathbb{Z}$.  It has a proper graded ideal $SCL_2(b,s)=\bigoplus_{i\neq -2b}\mathbb{C}[\partial]L_i \oplus \mathbb{C}[\partial](\partial+2s)L_{-2b}$. It is easy to see that $SCL_2(b,s)$ is simple as a $\mathbb{Z}$-graded Lie conformal algebra. 
	
	\section{Classification of simple $\mathbb{Z}$-graded Lie conformal algebras $\mathcal{L}$}
	
	In this section,  we will classify the following  simple $\mathbb{Z}$-graded  Lie conformal algebras $\mathcal{L}=\bigoplus_{i\in \mathbb{Z}} \mathcal{L}_i$ such that \begin{enumerate}
		\item $rank\mathcal{L}_i \leq 1$; 
		\item $\mathcal{L}_0$ is the Virasoro Lie conformal algebra.
	\end{enumerate}

	Set
	\[ Supp(\mathcal{L})=\{i \in \mathbb{Z}| rank\mathcal{L}_i=1\}.\] 
	If $\mathcal{L}\cong Vir$, then $Supp(\mathcal{L})=\{0\}$. In the sequel, we may focus on the case that  $|Supp(\mathcal{L})|>1$
	
	For each $i \in Supp(\mathcal{L})$, $\mathcal{L}_i$ can be regarded as a free rank one $\mathbb{C}[\partial]$-module with a $\mathbb{C}[\partial]$-basis $\{L_i\}$.  Hence for $i,j \in Supp(\mathcal{L})$,  $[{L_i}_\lambda L_j]=p_{i,j}(\partial,\lambda)L_{i+j}$ for some  $p_{i,j}(\partial,\lambda) \in \mathbb{C}[\partial,\lambda]$ (If $i+j\not\in Supp(\mathcal{L})$, we can identify $p_{i,j}(\partial,\lambda)$ with $0$). In addition,  we choose $L_0$ as a standard generator of the Virasoro Lie conformal algebra.  \cite[Lemma 5]{PK} implies that the set $S=\{i\in Supp(\mathcal{L})| p_{0,i}(\partial,\lambda)=0\}$ is empty otherwise $\bigoplus_{i\in S}\mathcal{L}_i$ will be a non-trivial proper ideal of $\mathcal{L}$. 
	Thus from \cite[Theorem 8.2]{DK},  for each $i\in Supp(\mathcal{L})$,  \[ p_{0,i}(\partial,\lambda)=\partial+a_i\lambda +b_i\] for some $a_i$,$b_i \in \mathbb{C}$.    Suppose that $i$, $j\in Supp(\mathcal{L})$ and $i+j\in Supp=(\mathcal{L})$.  By Jacobi-identity of $L_0,L_i,L_j$, the structure constants $p_{i,j}(\partial,\lambda)$ should satisfy the following equation
	\begin{equation}\label{esx1}\begin{aligned}
			(-\lambda-\mu+a_i\lambda+b_i)p_{i,j}(\partial,\lambda+\mu)=&p_{i,j}(\partial+\lambda,\mu)(\partial+a_{i+j}\lambda+b_{i+j})\\
			&-(\partial+\mu+a_j\lambda+b_j)p_{i,j}(\partial,\mu).
		\end{aligned}
	\end{equation} 
	
	One can see that the highest degree term of $p_{i,j}(\partial,\lambda)$, denoted by  $p^1_{i,j}(\partial,\lambda)$, is  determined by $a_i,a_j$ and  $a_{i+j}$.

	The Equation (\ref{esx1}) has been well studied in \cite{LY} as follows.
	
	\begin{p}{\label{p3.1}}If $p_{i,j}(\partial,\lambda)$ is some non-zero solutions of  Equation (\ref{esx1}), then 
		\begin{enumerate}
			\item[(i)] $a_i+a_j=a_{i+j}+\text{deg}\  p_{i,j}(\partial,\lambda)+1$.
			\item[(ii)] $b_i+b_j=b_{i+j}$.
		\end{enumerate}
		If $a_{i+j} \neq 0$, the nonzero homogeneous possible solution of  $p^{1}_{i,j}(\partial,\lambda)$ (up to scalars) is given as follows:
		\begin{enumerate}
			\item  $a_i\neq 1$.
			\begin{enumerate}
				\item $\text{deg}\  p^1_{i,j}(\partial,\lambda)=0$,\ $p^1_{i,j}(\partial,\lambda)=1$. 
				\item $\text{deg}\  p^1_{i,j}(\partial,\lambda)=1$,\ $p^1_{i,j}(\partial,\lambda)=\partial-\frac{a_{i+j}}{1-a_i}\lambda$.
				\item $\text{deg}\  p^1_{i,j}(\partial,\lambda)=2$,\ $a_j=1$, \ $p^1_{i,j}(\partial,\lambda)=(\partial+\lambda)(\partial-\frac{a_{i+j}}{1-a_i})$.
				\item $\text{deg}\  p^1_{i,j}(\partial,\lambda)=3$,\ $a_i=a_j=\frac{5}{3}$, \  $p^1_{i,j}(\partial,\lambda)=(\partial-\lambda)(\partial+2\lambda)(\partial+\frac{\lambda}{2}).$
			\end{enumerate}
			\item $a_i=1$.
			\begin{enumerate}
				\item $\text{deg}\  p_{i,j}(\partial,\lambda)=0$,   $p^1_{i,j}(\partial,\lambda)=1$.
				\item $\text{deg}\  p_{i,j}(\partial,\lambda)=1$,    $p^1_{i,j}(\partial,\lambda)=\lambda$.
				\item $\text{deg}\  p_{i,j}(\partial,\lambda)=2$,   $p^1_{i,j}(\partial,\lambda)=\lambda(\partial-a_{i+j}\lambda)$.
				\item $\text{deg}\  p_{i,j}(\partial,\lambda)=3$,  $a_j=1$,  $p^1_{i,j}(\partial,\lambda)=\lambda(\partial+\lambda)(\partial+2\lambda)$.
			\end{enumerate}
		\end{enumerate}
	\end{p}
	
	Indeed,  the case $a_{i+j}=0$ can be treated similarly. 
	As a consequence,  we have the following proposition.
	\begin{p}{\label{p1}} If $a_{i+j}=0$, the nonzero homogeneous possible solution(up to a scalar) of  Equation (\ref{esx1})  is given as follows:
		\begin{enumerate}
			\item $\text{deg}\  p^1_{i,j}(\partial,\lambda)=0$,\ $p^1_{i,j}(\partial,\lambda)=1$.
			\item $\text{deg}\  p^1_{i,j}(\partial,\lambda)=1$,\ $a_i\neq 1$ \text{or} $a_j\neq 1$,  $p^1_{i,j}(\partial,\lambda)=\partial$.
			\item $\text{deg}\  p^1_{i,j}(\partial,\lambda)=2$,\  $a_i\neq 1$, $p^1_{i,j}(\partial,\lambda)=\partial(\partial-\frac{1}{1-a_i}\lambda)$.
			\item  $\text{deg}\  p^1_{i,j}(\partial,\lambda)=2$,\  $a_i=1$,\ $a_j=2$,\  $p^1_{i,j}(\partial,\lambda)=\partial\lambda$.
			\item  $\text{deg}\  p^1_{i,j}(\partial,\lambda)=3$,\  $a_i=1$, $a_j=3$,  $p^1_{i,j}(\partial,\lambda)=\partial\lambda(\partial-\lambda)$.
			\item $\text{deg}\  p^1_{i,j}(\partial,\lambda)=3$,\ $a_i=3$,\  $a_j=1$, \  $p^1_{i,j}(\partial,\lambda)=\partial^3+\frac{3}{2}\partial^2\lambda+\frac{1}{2}\partial\lambda^2$.
		\end{enumerate}
	\end{p}
	\begin{proof}
		Let $\partial=-b_{i+j}$ in Equation (\ref{esx1}), we have
		\begin{equation}\label{rec1}
			(-\lambda-\mu+a_i\lambda+b_i)p_{i,j}(-b_{i+j},\lambda+\mu)=-(-b_{i+j}+\mu+a_j\lambda+b_j)p_{i,j}(-b_{i+j},\mu).
		\end{equation} 
		If $p_{i,j}(-b_{i+j},\mu)\not\in \mathbb{C}$, then the degree of $\lambda$ of both sides of  Equation {\eqref{rec1}} forces that  $a_i=a_j=1$.  In conclusion, if $a_{i}\neq 1$ or $a_j\neq 1$, then we have 
		\begin{equation}{\label{fe}}
			p_{i,j}(\partial,\lambda)=
			\begin{cases}
				(\partial+b_{i+j})f_{i,j}(\partial,\lambda),& 
				\text{if   $\text{deg}\   p_{i,j}(\partial,\lambda)\geq 1$}, \\
				w_{ij},&\text{else},
			\end{cases}
		\end{equation}
		for some $w_{ij}\in \mathbb{C}$ and   $f_{i,j}(\partial,\lambda)\in \mathbb{C}[\partial,\lambda]$.
		It is straightforward that $f_{i,j}(\partial,\lambda)$ is a  solution of Equation (\ref{esx1}) with $a_{i+j}=1$. The left can be obtained by Proposition  {\ref{p3.1}}.
	\end{proof}

	\begin{corollary}\label{c3.1}
		If $Re(a_i),\ Re(a_j)\geq 2$ and $p_{i,j}(\partial,\lambda)\neq 0$ for some $i,j\in Supp(\mathcal{L})$, then $\text{deg}\  p_{i,j}(\partial,\lambda)\leq 1$.
	\end{corollary}
	\begin{proof}
		It is immediate from    Propositions \ref{p3.1} and  \ref{rec1}.
	\end{proof}
	Set \[Supp(\mathcal{L})_1=\{k>0| k,-k\in Supp(\mathcal{L}), p_{k,-k}(\partial,\lambda)\neq 0\}.\] 
	If $\mathcal{L}\not\cong Vir$,   $Supp(\mathcal{L})_1$ is not empty otherwise  $\bigoplus_{i\neq 0}\mathcal{L}_i$ will be a proper ideal of $\mathcal{L}$.
	For each $k\in Supp(\mathcal{L})_1$, let $\mathcal{L}[k]$ be the  subalgebra of $\mathcal{L}$ generated by ${L_k,L_{-k}}$. We will investigate the structure of  $\mathcal{L}[k]$.
	For each $i\in \mathbb{N}$, set \[I_i=\{k\in Supp(\mathcal{L})_1|\text{deg}\   p_{k,-k}(\partial,\lambda)=i\}.\] 
	\begin{p}\label{rec2}
		$I_i={0}$ whenever $i\geq 3$. In other words, \[Supp(\mathcal{L})_1=I_0\cup I_1 \cup I_2.\]
	\end{p}
	\begin{proof}
		One should notice that $a_k+a_{-k}=\text{deg}\  p_{-k,k}(\partial,\lambda)+3$ as $a_0=2$. If $p_{-k,k}(\partial,\lambda)=3$,  then
		$a_k+a_{-k}=6$, which is impossible by Propositions 3.1 and 3.2. 
	\end{proof}	\begin{lemma}{\label{l3.1}} Suppose that $t \in Supp(\mathcal{L})$ and  $k \in Supp(\mathcal{L})_1$.
		\begin{enumerate}
			\item  $p_{t,t}(\partial,\lambda)$ is divisible by $\partial+2\lambda$.
			\item  If  $p_{k,t}(\partial,\lambda)=0$ or  $p_{-k,t}(\partial,\lambda)=0$, then  $a_t=0$ or 1.
			\item  If $k \in I_i$ and $p_{k,t}(\partial,\lambda),p_{k+t,-k}(\partial,\lambda)\neq 0$, then  \[ \text{deg}\  p_{k,t}(\partial,\lambda)+\text{deg}\  p_{k+t,-k}(\partial,\lambda)=1+i.\]  
			\item If  $\text{deg}\  p_{-k,k}(\partial,\lambda)>0$ and $Re(a_k)\geq Re(a_{-k})$,  then $ p_{k,jk}(\partial,\lambda)\neq 0$ and $\text{deg}\  p_{k,jk}(\partial,\lambda)\leq 1$ for each $j>0$. Moreover, $Re(a_{uk})\geq Re(a_{vk})\geq 2$ for any $u>v\geq 0$.
			\item  Suppose that $k\in I_1$  and $Re(a_k)\geq Re(a_{-k})$. If $a_{-k}\neq 1$, then  $k\mathbb{Z}\subset Supp(\mathcal{L})$ and $\text{deg}\  p_{-k,jk}(\partial,\lambda)=1$ for each $j\geq 0$. Moreover, $p_{ik,jk}(\partial,\lambda)\neq 0$ for any $i<0$ and $j>0$.  
		\end{enumerate}
	\end{lemma}
	\begin{proof}
		
		$(1)$ By skew-symmetric property, we have  \[p_{t,t}(\partial,\lambda)=-p_{t,t}(\partial,-\partial-\lambda).\] Plugging $\partial=-2\lambda$ into above equation, we can see that  $p_{t,t}(-2\lambda,\lambda)=0$. Hence $p_{t,t}(\partial,\lambda)$ has a factor $\partial+2\lambda$.

		$(2)$  If  $p_{-k,t}(\partial,\lambda)=0$,  from the  following Jacobi-identity
		\[ [[{L_{-k}}_\lambda L_{k}]_{\lambda+\mu} L_t]=-[{L_{-k}}_\lambda [{L_{k}}_\mu L_t]],\]
		we have 
		\begin{equation}\label{e3.4} 
			p_{-k,k}(-\lambda-\mu,\lambda)(\partial+a_t\lambda+a_t\mu+b_t)=p_{k,t}(\partial+\lambda,\mu)p_{-k,t+k}(\partial, \lambda).
		\end{equation}
		Thus $\partial+a_t\lambda+a_t\mu+b_t$ is a factor of either $p_{k,t}(\partial+\lambda,\mu)$ or $p_{-k,t+k}(\partial, \lambda)$ since $\mathbb{C}[\partial,\lambda]$ is a unique  factorization domain. Hence $a_t=1$ or $0$.  It is similar for the case that $p_{k,t}(\partial,\lambda)=0$  .

		$(3)$ By  Proposition {\ref{p3.1}}$(i)$, we have \[a_k + a_t = a_{k+t} + \text{deg}\   p_{k,t}(\partial,\lambda)  + 1,\]
		\[a_{k+t} + a_{-k} = a_t + \text{deg}\   p_{k+t,-k}(\partial,\lambda)  + 1,\]
		\[a_k + a_{-k} =2+ \text{deg}\   p_{k, -k}(\partial,\lambda) + 1.\]
		Hence \[\text{deg}\   p_{k, t}(\partial,\lambda) + \text{deg}\   p_{k+t, -k}(\partial,\lambda) = a_k + a_{-k} - 2 = \text{deg}\   p_{k, -k}(\partial,\lambda) + 1.\]

		$(4)$ Notice that \[a_k + a_{-k} =2+ \text{deg}\   p_{k, -k}(\partial,\lambda) + 1\geq 4.\] Hence $Re(a_k)\geq 2$. And by $(2)$,  $p_{k,k}(\partial, \lambda)\neq 0$. Then by $(1)$ and Corollary \ref{c3.1}, we can see that $\text{deg}\   p_{k,k}(\partial, \lambda)=1$.  Assume that there exists some $j_0>0$ such that $p_{k,jk}(\partial,\lambda)\neq 0$ for each $1\leq j<j_0$ and $p_{k,j_0k}(\partial,\lambda)=0$. Then
		for each $1\leq j<j_0$, we have
		\[a_k + a_{jk} =a_{(j+1)k}+ \text{deg}\   p_{k, jk}(\partial,\lambda) + 1.\]
		Using induction on $j$, together with Corollary \ref{c3.1}, we have $Re(a_{uk})\geq Re(a_{vk})\geq 2$  for each $j_0\geq u>v\geq 1$ and $\text{deg}\   p_{k,jk}(\partial,\lambda)\leq 1$ for each $1\leq j<j_0$. In particular, $Re(a_{j_0k})\geq 2$. However, from (2), $a_{j_0k}=0$ or $1$ for $p_{k,j_0k}(\partial,\lambda)=0$, which is a contradiction. Thus $(4)$ holds.

		$(5)$ 
		First of all, notice that for each $s\geq 0$, $\text{deg}\   p_{k,sk}(\partial,\lambda)\leq 1$ by $(4)$. Suppose that there exists some $s_0>0$ such that $\text{deg}\    p_{k,s_0k}(\partial,\lambda)=0$. Then by $(3)$  $\text{deg}\   p_{-k,(s_0+1)k}(\partial,\lambda)=2$. 
		By Propositions \ref{p3.1} and \ref{p1}, it happens only when  $a_{-k}=1$, which is a contradiction.  Hence  $\text{deg}\   p_{k,sk}(\partial,\lambda)=1$ for each $s\geq 0$.   By $(3)$, we have \[ \text{deg}\   p_{-k,(s+1)k}(\partial,\lambda)=2-\text{deg}\   p_{k,sk}(\partial,\lambda)=1\] 
		for each $s\geq 0$.

		Next we shall prove that $p_{-k,-tk}(\partial, \lambda)\neq 0$ for any $t>0$. Indeed, Suppose that there exists a maximal $t_0<0$ such that  $ p_{-k,t_0k}(\partial, \lambda)=0$. Thus $p_{-k,tk}(\partial, \lambda)\neq 0$ for each $t>t_0$.  In this case, we claim that $p_{ik,jk}(\partial,\lambda)\neq 0$ for any $t_0\leq i<0$ and $j>0$.  Otherwise, let  $i_0\geq t_0$ be the maximal negative integer such that $p_{i_0k, jk}(\partial, \lambda)=0$ for some positive integer $j$.   From the Jacobi-identity of $L_{-k}$,$L_{(i_0+1)k}$ and $L_{jk}$, we obtain that
		\[  [{L_{-k}}_\lambda [{L_{(i_0+1)k}}_\mu L_{jk}]]=[{L_{(i_0+1)k}}_\mu [{L_{-k}}_\lambda L_{jk}]],\]
		Thus
		\begin{equation}{\label{bc}}p^1_{(i_0+1)k,jk}(\partial+\lambda,\mu)p^1_{-k,(i_0+j+1)k}(\partial,\lambda)=p^1_{-k,jk}(\partial+\mu,\lambda)p^1_{(i_0+1)k,(j-1)k}(\partial,\mu)
		\end{equation}
		Notice that $\text{deg}\    p^1_{-k,jk}(\partial,\lambda)=1$. Hence by
		Propositions 3.1,  \[p^1_{-k,jk}(\partial,\lambda)=c_{-k,jk}(\partial+\frac{a_{(j-1)k}}{a_{-k}-1}\lambda)\] for some  $c_{-k,jk}\in \mathbb{C}$. Since $p_{-k,jk}(\partial,\lambda)$ is irreducible, Equation (\ref{bc}) forces that $p_{(i_0+1)k,jk}(\partial+\lambda,\mu)$ is divisible by $p_{-k,jk}(\partial+\mu,\lambda)$.     Since $p_{(i_0+1)k,(j-1)k}(\partial,\lambda)\neq 0$, one can see that $\frac{a_{(j-1)k}}{a_{-k}-1}=1$. Hence
		$Re(a_{(j-1)k})=Re(a_{-k})-1\leq 1$, which is impossible by $(4)$. Thus we finish the proof of claim.  
		
		From the claim, we can see that $-t_0k\in Supp(\mathcal{L})_1$. Since  $ p_{-k,t_0k}(\partial, \lambda)=0$,  $a_{-k}=0$ and $a_{t_0k}=0$ or $1$ by $(2)$. Hence Equation {\eqref{e3.4}} implies that $p_{-k,-k}(\partial, \lambda)\neq 0$ by taking $t=-k$. Then by Proposition {\ref{p3.1}} and $(1)$, $\text{deg}\  p_{-k,-k}(\partial, \lambda)=1$. Thus $a_{-2k}=-2$. It forces that $a_{t_0k}\leq -2$, which is impossible. The contradiction means that  $ p_{-k,-tk}(\partial, \lambda)\neq 0$ for each $t>0$.  Hence $k\mathbb{Z}\subset Supp(\mathcal{L})$. The proof of claim also illustrates that $p_{ik, jk}(\partial,\lambda)\neq 0$ for each $i<0$ and $j>0$.  
	\end{proof}

	\begin{p}{\label{l2}}Suppose that  $k \in I_1\cup I_2$ and $Re(a_{-k})\leq Re(a_{k})$. If $a_{-k}=1$, then $\mathcal{L}[k]$ must be one of followings:
		\begin{enumerate}
			\item  $p_{-k,-k}(\partial,\lambda)=0$ and  $L[k]\cong CL_1(s)$ for some $s\in \mathbb{C}$.  
			\item $\text{deg}\  p_{-k,-k}(\partial,\lambda)=3$ and  $\mathcal{L}[k]\cong SCL_2(-1,s)$.
			\item $\text{deg}\  p_{-k,-k}(\partial,\lambda)=1$ and $\mathcal{I}:=\bigoplus_{i\leq -2}\mathcal{L}_{ik}\cap \mathcal{L}[k]$ is an ideal of $\mathcal{L}[k]$. Furthermore, $\mathcal{L}/\mathcal{I}\cong  CL_1(s)$ for some $s\in \mathbb{C}$.
		\end{enumerate}
		In particular, $a_{jk}=a_k+(j-1)(a_k-2)$ for each $j\geq 0$.
	\end{p}
	\begin{proof}
		First notice that by Lemma 3.1(1),  $\partial+2\lambda$ is a factor of $p_{-k,-k}(\partial,\lambda)$. Suppose that $p_{-k,-k}(\partial,\lambda)\neq 0$. If $a_{-2k}=0$, then by Proposition 3.1(i), we have \[\text{deg}\  p_{-k,-k}(\partial,\lambda)=1+1-1=1.\] If $a_{-2k}\neq 0$, then $\text{deg}\  p_{-k,-k}(\partial,\lambda)=3$ by  Proposition 3.1(2). Thus we distinguish the following three cases.

		$(1)$$p_{-k,-k}(\partial, \lambda)=0$. 	
		Since $Re(a_{jk})\geq 2$,  for each $j\geq 0$, $p_{-k,jk}(\partial,\lambda)\neq 0$ from  Lemma \ref{l3.1}(2).
		Now from Lemma \ref{l3.1}(3)(4),  one can see that $2\geq \text{deg}\  p_{-k,jk}(\partial,\lambda)\geq 1$ for each $j\geq 0$. Suppose that $\text{deg}\  p_{-k,j_0k}(\partial,\lambda)=2$ for some $j_0>0$. Then by Proposition 3.1,  \[p^1_{-k,j_0k}(\partial,\lambda)=c_{-k,j_0k}\lambda(\partial-a_{(j_0-1)k}\lambda)\]for some $c_{-k,j_0k}\in \mathbb{C}^*$. Plugging it into the Jacobi-identity of $L_{-k}$,$L_{-k}$ and $L_{j_0k}$, we can find that $a_{(j_0-1)k}=-1$, which is impossible. Hence from Lemma \ref{l3.1}(3), $\text{deg}\ p_{-k,jk}(\partial,\lambda)=\text{deg}\ p_{k,jk}(\partial,\lambda)=1$ for each $j\geq 0$. Thus by Proposition 3.1$(i)$, $a_{jk}=2+j$ for each $j\geq -1$. Combining with Corollary 3.1,    $\mathcal{L}[k]=\bigoplus_{i\geq -1}\mathbb{C}[\partial]L_{ik}$ is quadratic and $(\mathcal{L}[k])_N$ is the direct sum of eigenspace with distinct positive eigenvalues of  $L_0$. By the action of $L_{-k}$, one can see that  $(\mathcal{L}[k])_N$ is simple. Hence, by  Proposition \ref{OP}, $(\mathcal{L}[k])_N$ is isomorphic to $A_1$. Thus  $\mathcal{L}[k]\cong CL_1(s)$ for some $s\in \mathbb{C}$ by Proposition \ref{XP}.

		$(2)$ $\text{deg}\  p_{-k,-k}(\partial, \lambda)=3$. Thus $a_{-2k}=-2$.  In addition, by Lemma 3.1(4), $a_{ik}\geq 4$ for each $i\geq 1$. Considering the considering the Jacobi-identity of $L_{-k}$,$L_{-k}$ and $L_{k}$, coefficients before $\partial^2$ in Equation (\ref{e3.4}) forces that $deg p_{-k,k}(\partial, \lambda)=2$. Hence $a_k=4$.  By using induction on $j$, one is easy to see that $p_{-k,-jk}(\partial,\lambda)\neq 0$ and $a_{-jk}\leq -2$ for each $j\geq 1$. Hence 
		$\text{deg}\  p_{-k,uk}(\partial,\lambda) \leq 1$ and 	$\text{deg}\  p_{k,uk}(\partial,\lambda) \leq 2$  for each $u\neq -1$. Thus Lemma 3.1(3) and (4) force that 	$\text{deg}\  p_{k,uk}(\partial,\lambda)=1$ for each $u\neq -1$ and 	$\text{deg}\  p_{-k,vk}(\partial,\lambda)=2$ for each $v\neq -1$ or $0$. 
		It implies that $a_{uk}=2u+2$ for each $u\neq -1$. Thus $p_{uk,(-1-u)k}(\partial, \lambda)\neq 0$ for each $u \neq -1$ or $0$.   Hence, for each $u \neq -1$ or $0$,  $\text{deg}\  p_{uk,(-1-u)k}(\partial, \lambda)=0$ by Proposition 3.1(i).  For each $t\in \mathbb{Z}$,   the Jacobi-identity of $L_0$,$L_{-k}$ and $L_{tk}$ implies that 
		\begin{equation}{\label{ex2}}
			\begin{aligned}
				&(-\mu+b_{-k})p_{-k,tk}(\partial,\lambda+\mu)\\=&p_{-k,tk}(\partial+\lambda,\mu)(\partial+a_{(t-1)k}\lambda+b_{(t-1)k})\\
				&-(\partial+\mu+a_{tk}\lambda+b_{tk})p_{-k,tk}(\partial,\mu).
			\end{aligned}
		\end{equation}
		Let $\mu=b_{-k}$  in Equation (\ref{ex2}). We obtain that 
		\begin{equation}{\label{ex33}}
			p_{-k,tk}(\partial+\lambda,b_{-k})(\partial+a_{(t-1)k}\lambda+b_{(t-1)k})=
			(\partial+b_{-k}+a_{tk}\lambda+b_{tk})p_{-k,tk}(\partial,b_{-k}).
		\end{equation} 
		Notice that $a_{(t-1)k}\neq 0$ for each $i\in \mathbb{Z}$.  Since $\text{deg}\  p_{-k,tk}(\partial, \lambda)\geq 1$, the coefficients before $\lambda$ of Equation {\eqref{ex33}} forces that $p_{-k,tk}(\partial,b_{-k})=0$. Therefore $-\lambda+b_{-k}$  is a factor of $p_{-k,tk}(\partial,\lambda)$ for each $t\in \mathbb{Z}$. By skew-symmetric property that
		$p_{-k,-k}(\partial,\mu)$ is also divisible by $\partial+\lambda+b_{-k}$.
		Define a free $\mathbb{C}[\partial]$-module  $\mathcal{T}=\bigoplus_{i\in \mathbb{Z}}\mathbb{C}[\partial]T_i$ and  $\lambda$-brackets on  $\mathcal{T}$ is given as follows: 
		\[  [{T_i}_\lambda T_{-1-i}]=(\partial+b_{-k})p_{ik,(-1-i)k}(\partial, \lambda)T_{-1}, \text{for each}\  i \in \mathbb{Z},\]
		\[ [{T_i}_\lambda T_j]=p_{ik,jk}(\partial,\lambda)T_{i+j}, \text{for}\   i,j\neq -1\  \text{and}\  i+j\neq -1,\]
		\[  [{T_{-1}}_\lambda T_j]=\frac{ p_{-k, jk}(\partial,\lambda)}{(-\lambda+b_{-k})}T_{-1+j}, \text{for} \ j\neq 0\  \text{and}\  j\neq -1,\]
		\[   [{T_{-1}}_\lambda T_{-1}]=\frac{p_{-k,-k}(\partial,\lambda)}{(-\lambda+b_{-k})(\partial+\lambda+b_{-k})}T_{-2},  \text{for}\   j\neq 0.\]
		One can check directly that $\mathcal{T}$ 
		is quadratic and is isomorphic to $CL_2(\frac{1}{2},s)$ for some $s\in \mathbb{C}$. On the other hand, $\mathcal{L}[k]$ can be imbeded into $\mathcal{T}$ by mapping $L_{ik}$ to $T_i$ for $i\neq -1$ and mapping $L_{-k}$ to $(\partial+b_{-k})T_{-1}$.  Hence  $\mathcal{L}[k]\cong SCL_2(\frac{1}{2},s)$ for some $s\in \mathbb{C}$.
		
		$(3)$ $\text{deg} p_{-k,-k}(\partial,\lambda)=1$. In this case,  $a_{-k}=1$ and $a_{-2k}=0$.  The Jacobi-identity of  $L_{-k}, L_{-k}, L_{k}$ forces that 	\[(\lambda-\mu)p^1_{-2k,k}(\partial,\lambda+\mu)=c\lambda\mu(\lambda-\mu)=0.\] 
		for some $c\in \mathbb{C}$. It implies that 
		$p_{-2k,k}(\partial,\lambda)=0$ for each $j\geq 0$.  Then by using indcution on $i$, one is easy to see that $p_{-ik,jk}(\partial,\lambda)=0$ for each $0<i-1\leq j$. The left is obvious.
	\end{proof}
	
	\begin{lemma}{\label{l3.2}}
		Suppose that $k\in I_1 \cup I_2$ and  $Re(a_{-k})\leq Re(a_k)$. Then \begin{enumerate}
			\item[(i)]  $a_{jk}=a_{k}+(j-1)(a_{k}-2)$ for each  $j\geq 0$.
			\item[(ii)] $|k\mathbb{Z}\cap I_0|\leq 1$ and $|k\mathbb{Z}\cap I_2|\leq 1$. 
		\end{enumerate} 
	\end{lemma}
	\begin{proof}
		Since $Re(a_{-k})\leq Re(a_k)$  and $a_{k}+a_{-k}\geq 4$,  we have $Re(a_k)\geq 2$.  	By Proposition {\ref{l2}}, the lemma holds wherenever  $a_{-k}=1$. So it is sufficient to consider the case that $a_{-k}\neq 1$. Thus we may assume that $k\in I_1$.
		By Lemma \ref{l3.1}(5),  $k\mathbb{Z}\subset Supp(\mathcal{L})$ and $jk\in Supp_1(\mathcal{L})$ for each $j>0$.	
		
		$(i)$  By Lemma {\ref{l3.1}}(3) and (5), $\text{deg}\  p_{-k,jk}(\partial,\lambda)=1$ and  $\text{deg}\  p_{k,jk}(\partial,\lambda)=1$ for each $j\geq 0$.   By Proposition \ref{p3.1}(i), $a_{jk}=a_{k}+(j-1)(a_{k}-2)$ for each  $j\geq -1$.

		$(ii)$Suppose that there exists some  $sk,tk \in k\mathbb{Z}\cap I_2$. Then  $a_{-sk}=a_{-tk}=1$ by Propositions 3.1 and 3.2. Hence $a_{sk}=a_{tk}=4$.   It happens only when  $s=t$ by $(i)$. Hence  $|k\mathbb{Z}\cap I_2|\leq 1$. 
		
		Suppose that there exists some $s\in \mathbb{Z}$ such that $sk\in I_0$. Let $s_0$ be the minimal of such one. Obviously, $s_0>1$. 
		Then $(s_0-1)k\in I_1\cup I_2$.
		Suppose that $(s_0-1)k\in I_2$. Then $a_{(-s_0+1)k}=1$.
		Notice that
		\[a_{-k}+a_{(1-s_0)k}=a_{-s_0k}+\text{deg}\  p_{-k,(1-s_0)k}(\partial,\lambda)+1,\]
		\[a_{k}+a_{(s_0-1)k}=a_{s_0k}+\text{deg}\  p_{k,(s_0-1)k}(\partial,\lambda)+1.\]
		Thus $\text{deg}\  p_{-k,(1-s_0)k}(\partial,\lambda)=3$, which is impossible for $a_{(1-s_0)k}=1$ and $a_{-k}\neq 1$.  Hence  $(s_0-1)k\in I_1$.
		Thus
		\[\text{deg}\  p_{-k,(1-s_0)k}(\partial,\lambda)=a_{-k}+a_{(1-s_0)k}-1-a_{-s_0k}=a_{-k}-a_{(s_0-1)k}+a_{s_0k}=2.\]
		Hence $a_{(1-s_0)k}=1$ or $a_{-s_0}=0$. Thus $a_{(s_0-1)k}=3$ or $a_{s_0k}=3$. Hence $a_{-s_0k}\leq 0$ and $a_{-k}<2$ in both cases.  In addition, 
		\[\text{deg}\  p_{-s_0k,2s_0k}(\partial,\lambda)=a_{-s_0k}+a_{2s_0k}-1-a_{s_0k}=a_{-s_0k}+2a_{s_0k}-3-a_{s_0k}=0.\]
		Hence the  Jacobi-identity of $L_{-s_0k},L_{-s_0k}$ and $L_{2s_0k}$ implies that $p_{-s_0k,-s_0k}(\partial,\lambda)=0$. 
		By considering the Jacobi-identity of $L_{-k}$,$L_{(1-s_0)k}$ and $L_{-s_0k}$, 
		we can obtain that
		\begin{equation}{\label{bcc}}p_{(1-s_0)k,-s_0k}(\partial+\lambda,\mu)p_{-k,(1-2s_0)k}(\partial,\lambda)=p_{-k,-s_0k}(\partial+\mu,\lambda)p_{(1-s_0)k,(-s_0-1)k}(\partial,\mu)
		\end{equation}
		Notice that $\text{deg}\  p_{-k,-s_0k}(\partial+\lambda,\mu)\leq 1$. Otherwise $a_{(-1-s_0)k}+a_{(1+s_0)k}\leq 2$, which is impossible. Suppose that $\text{deg}\  p_{-k,-s_0k}(\partial+\lambda,\mu)=1$. Then $a_{(-1-s_0)k}=a_{-k}+a_{-s_0k}-2<0$ for $a_{-k}<2$. It implies that  $p_{(1-s_0)k,(-s_0-1)k}(\partial,\mu)\neq 0$. Since \[p^1_{-k,-s_0k}(\partial+\mu,\lambda)=c_{-k,k}(\partial+\mu+\frac{a_{(-s_0-1)k}}{a_{-k}-1})\lambda\] for some $c_{-k,k}\in \mathbb{C}^*$. Equation \eqref{bcc} forces that $a_{(-s_0-1)k}=a_{-k}-1$.
		It implies that $a_{-s_0k}=1$, which is a contradiction.  Hence   $\text{deg}\  p_{-k,-s_0k}(\partial,\lambda)=0$. Thus  
		$(1+s_0)k\in I_1$. Suppose that $S:=\{jk\in I_0|j>s_0\}$ is not empty. Let $s_1$ be the least integer of $S$. Clearly, $s_1>1+s_0$.  With a similar analysis as above,  one has $(s_1-1)k \in I_1$. Further,
		$a_{(s_1-1)k}=3$ or $a_{s_1k}=3$. It is contradict to the statement $(i)$. Thus $S$ is an empty set. It is equivalent to saying that 
		$|k\mathbb{Z}\cap I_0|\leq 1$.     
	\end{proof}

	\begin{p}{\label{bcp}} Suppose that $k\in I_1 \cup I_2$.  If $a_{-k}\neq 1$, then $\mathcal{L}[k]\cong CL_2(b,s)$ or $SCL_2(b,s)$ or $\mathcal{V}(s)$ for some $b,s\in \mathbb{C}$.
	\end{p}
	\begin{proof}we may assume that  $Re(a_k)\geq Re(a_{-k})$. Clearly, $k\in I_1$. Set $\mathcal{J}[k]:=\bigoplus_{i \in \mathbb{Z}}\mathbb{C}[\partial]L_{ik}$.
		By Lemma 3.2, we shall distinguish the following three cases.

		{\bf Case I}.$|k\mathbb{Z}\cap I_0|=0$ and $|k\mathbb{Z}\cap I_2|=0$.
		
		By Lemma 3.1(5), for each $j>0$, we have $jk\in I_1$. Together with Lemma {\ref{l3.2}}$(i)$, one can obtain that  $a_{jk}=a_{k}+(j-1)(a_{k}-2)$ for each  $j\in \mathbb{Z}$.  Hence by Proposition 3.1(i), $\mathcal{J}[k]$ is quadratic.
		
		Suppose that $a_k=2$. Then $a_{jk}=2$ for each $j\in \mathbb{Z}$. Moreover, for each $i,j\in \mathbb{Z}$, $\text{deg}\  p_{ik,jk}(\partial,\lambda)=1$.
		Thus for each $i,j\in \mathbb{Z}$, we may assume that \[p_{ik,jk}(\partial,\lambda)=c_{ik,jk}(\partial+2\lambda+t_{ik,jk})\] for some  $c_{ik,jk}, t_{ik,jk}\in \mathbb{C}$ by Propositions \ref{p3.1}. Then for each $i,j,x\in \mathbb{Z}$,  Considering the Jacobi-identity of $L_{ik},L_{jk},L_{xk}$, we obtain
		\begin{equation}\label{bc22}
			\begin{aligned}
				&c_{(i+j)k,sk}c_{ik,jk}(\lambda-\mu+t_{ik,jk})(\partial+2(\lambda+\mu)+t_{(i+j)k,xk})\\=&c_{jk,xk}c_{ik,(j+x)k}(\partial+\lambda+2\mu+t_{jk,xk})(\partial+2\lambda+t_{ik,(j+x)k})\\&-c_{ik,xk}c_{jk,(i+x)k}(\partial+2\lambda+\mu+t_{ik,xk})(\partial+2\mu+t_{jk,(i+x)k}).
			\end{aligned}
		\end{equation}
		The coefficients before $\lambda^2$  in Equation (\ref{bc22}) require that
		\begin{equation}{\label{bdc}} c_{(i+j)k,xk}c_{ik,jk}=c_{jk,xk}c_{ik,(j+x)k}.
		\end{equation}
		Since 	$\mathcal{L}[k]$ is generated $L_{k}$,	$L_{-k}$, we can assume that \[c_{k,jk}=c_{-k,-jk}=c_{-k,k}=1.\] for each $j\geq 0$. 
		By letting $i=1$ and $j=-1$ in Equation \eqref{bdc}, we can find that
		\[1=c_{k,xk}c_{-k,(1+x)k}.\]
		for each $x\in \mathbb{Z}$. Hence
		\[c_{k,xk}=c_{-k,xk}=1\]
		for each $x\in \mathbb{Z}$. Then plugging $i=1$ into Equation \eqref{bdc} and using induction on $j$, one can see that $c_{jk,xk}=1$ for each $j,x\in \mathbb{Z}$.			
		Now  coefficients of $\partial$ of Equation \eqref{bc22} imply that
		\begin{equation}\label{bc11}
			t_{ik,jk}=t_{jk,xk}+t_{ik,(j+x)k}-t_{ik,xk}-t_{jk,(i+x)k}.
		\end{equation}	
		By Proposition {\ref{p3.1}}(ii),  $t_{ik,0}=-t_{0,ik}=-b_{ik}=-ib_{k}$ for each $i\in \mathbb{Z}$.    Take $x=0$ in  Equation (\ref{bc11}), \[       t_{jk,ik}=t_{jk,0}-t_{ik,0}=-(j-i)b_{k}.\]    One can check directly that  $\mathcal{L}[k]\cong  \mathcal{V}[-b_k]$.
		
		Next, we shall deal with the case that $a_{k}\neq 2$. In this case, the values of $a_{jk}$ are different for each $j\in \mathbb{Z}$. Hence the corresponding Novikov algebra $(\mathcal{J}[k])_N$ is the direct sum of eigenspace with distinct eigenvalues of  $L_0$. Hence any ideal of $(\mathcal{J}[k])_N$ must be graded. Notice that for each $i\in \mathbb{Z}$, either $L_{-k}\circ L_{jk}\neq 0$ or $L_{jk}\circ L_{-k}\neq 0$. Hence $(\mathcal{J}[k])_N$ is simple. Consequently,  $\mathcal{J}[k]\cong CL_2(u,s)$ for some  $u,s\in \mathbb{C}$. 
		Clearly, if $2u\not\in \mathbb{Z}$, then $\mathcal{L}[k]=\mathcal{J}[k]$. Suppose that  $2u\in \mathbb{Z}$. Then $a_{-2uk}=0$. Since $\text{deg}\ p_{ik,(-i-2u)k}=1$ for each $i\in \mathbb{Z}$, 
		one can see that  $L_{-2uk}\not\in \mathcal{L}[k]$ but $(\partial+b_{-2uk})L_{-2uk}\in \mathcal{L}[k]$.  Hence
		\[\mathcal{L}[k]=\bigoplus_{i\neq -2b} \mathbb{C}[\partial]L_{ik}\oplus \mathbb{C}[\partial](\partial+b_{-2k})L_{-2k}\]
		is a peoper ideal of  $\mathcal{J}[k]$.  Hence, in this case, $\mathcal{L}[k]\cong SCL_2(u,s)$.
		
		{\bf Case II}. $|k\mathbb{Z}\cap I_0|=0$ and $|k\mathbb{Z}\cap I_2|=1$.

		Assume that $s_1k\in I_2$ for some $s_1>0$. Then $a_{-s_1k}=1$.  Together with Lemma {\ref{l3.2}}(i), one can obtain that  $a_{jk}=a_{k}+(j-1)(a_{k}-2)$ for each  $j\in \mathbb{Z}$ and $j\neq -s_1$. Thus if $p_{ik,jk}(\partial,\lambda)\neq 0$ for some
		$i\leq j$ , then
		\begin{align*}
			\begin{split} 
				\text{deg}\  p_{ik,jk}(\partial,\lambda)=\left\{ 
				\begin{array}{ll}
					0,& i+j=-s_1\  \text{and}\  j\neq 0 ,\\
					2,&i=-s_1\  \text{and}\  j\neq 0,\ -s_1,\\
					3,&i=j=-s_1,\\
					1,&else.
				\end{array}
				\right.
			\end{split}
		\end{align*}
		Then using a similar proof as in (2) of Proposition \ref{l2}, one can see that $\mathcal{L}[k]\cong SCL_2(b,s)$ for some $2b\in \mathbb{Z}$ and $s \in \mathbb{C}$.

		{\bf Case III}. $|k\mathbb{Z}\cap I_0|=1$.

		Assume that $s_0k\in I_0$ for some $s_0>0$. 
		Notice that for large enough $t_0$, $Re(a_{t_0k})>4$ and $t_0k\in I_1$. Thus $(s_0+t_0)k\in I_1$ and $Re(a_{(-s_0-t_0)k})<0$. Thus $p_{t_0k,(-s_0-t_0)k}(\partial,\lambda)\neq 0$. Hence
		\[\text{deg}\  p_{t_0k,(-s_0-t_0)k}=a_{t_0k}+a_{(-s_0-t_0)k}-a_{-s_0k}-1=
		a_{t_0k}-a_{(s_0+t_0)k}+a_{s_0k}=2.\]
		It forces that $a_{-s_0k}=0$ and $a_{s_0k}=3$. Since $a_{(-s_0-t)k}\geq  3-a_{(s_0+t)k}$ for each $t\in \mathbb{Z}$, 
		\[\text{deg}\  p_{tk,(-s_0-t)k}=a_{tk}+a_{(-s_0-t)k}-a_{-s_0k}-1\geq
		a_{tk}+(3-a_{(s_0+t)k})-(3-a_{s_0k})-1=1.\]
		Notice that for each $t_0>t>0$, $a_{-t}\geq 4-a_t>1$. 
		Hence for each $t\in \mathbb{Z}$, either $t$ or $-t_0-t$ is not equal to $1$. Hence by  Equation (\ref{fe}), 
		one can see that $\partial+b_{-s_0k}$ is a  factor of each  $ p_{(-s_0-t)k,tk}(\partial,\lambda)$ for each $t\in \mathbb{Z}$. It implies that $L_{-s_0k}\not\in \mathcal{L}[k]$ but $(\partial+b_{-s_0k})L_{-s_0k}\in \mathcal{L}[k]$. Hence \[\mathcal{L}[k]=\bigoplus_{i\neq -s_0}\mathbb{C}[\partial]L_{ik}\oplus \mathbb{C}[\partial](\partial+b_{-s_0k})L_{-s_0k}.\]
		Hence $\mathcal{L}[k]$ admits a $\mathbb{C}[\partial]$-basis $\{L'_{i}, i\in \mathbb{Z}\}$, where $L'_{-s_0}=(\partial+b_{-s_0k})L_{-s_0k}$ and $L'_{j}=L_{jk}$ for other $j\neq -s_0$. By considering the structure constants in terms of the $\mathbb{C}[\partial]$-basis $\{L'_{i}, i\in \mathbb{Z}\}$, {\bf Case III} can be reduced to {\bf Case I} for $|k\mathbb{Z}\cap I_2|=0$. In addition, {\bf Case III} can be reduced to {\bf Case II} for $|k\mathbb{Z}\cap I_2|=1$. 			 
		
	\end{proof}

	Next, we shall deal with the case that $k\in I_0$.
	
	\begin{p} Suppose that $k\in I_0$. If $Re(a_{-k})\leq Re(a_{k})$, then we have $p_{-k,-k}(\partial,\lambda)=0$. \end{p}
	\begin{proof}
		Suppose that $p_{-k,-k}(\partial,\lambda)\neq 0$. Consider the Jacobi-identity of $L_{-k},L_{-k}$ and $L_{k}$. If $a_{-k}=1$, then  we have
		\[p_{-k,-k}(-\lambda-\mu,\lambda)p_{-2k,k}(\partial,\lambda+\mu)=t_{-k,k}(\lambda-\mu),\]
		for some $t_{-k,k} \in \mathbb{C}^*$. Thus $p_{-2k,k}(\partial,\lambda)\in \mathbb{C}^*$ and $deg p_{-k,-k}(\partial,\lambda)=1$.
		If  $a_{-k}\neq 1$, then   
		\[ p_{-k,-k}(-\lambda-\mu,\lambda)p_{-2k,k}(\partial,\lambda+\mu)=c_{-k,k}\frac{a_{-k}}{a_{-k}-1}(\lambda-\mu)\]
		for some $c_{-k,k} \in \mathbb{C}$. Since $\text{deg}\  p_{-k,-k}(\partial,\lambda)\geq 1$, it  implies that $p_{-2k,k}(\partial,\lambda)=0$ if and only if  $a_{-k}=0$. However, $a_{-k}=0$ means that $a_{-2k}=-2$, which is contradict to Lemma 3.1(2).  Hence $p_{-2k,k}(\partial,\lambda)\in \mathbb{C}^*$ and $deg p_{-k,-k}(\partial,\lambda)=1$ in both cases. Similarly, we can see that  $p_{2k,-k}(\partial,\lambda)\in \mathbb{C}^*$.   Thus $a_{2k}+a_{-2k}=2a_{k}+2a_{-k}-4=2$. Hence $p_{2k,-2k}(\partial,\lambda)=0$. Thus
		\[ [{L_{-2k}}_\lambda [{L_{2k}}_\mu L_k]]=[{L_{2k}}_\mu [{L_{-2k}}_\lambda L_k]]\in \mathbb{C}^*L_k. \]
		Hence $p_{2k,k}(\partial,\lambda), p_{-2k,3k}(\partial,\lambda)\in \mathbb{C}^*$.   Similarly, we have $p_{-2k,-k}(\partial,\lambda)\in \mathbb{C}^*$.   Hence $a_{3k}=3a_{k}-3$ and $a_{-3k}=3a_{-k}-3$.
		Moreover, from the Jacobi-identity of $L_{-k},L_{-2k}$ and $L_{3k}$, one can see that  
		\[ [{L_{-k}}_\lambda [{L_{-2k}}_\mu L_{3k}]]=[{[{L_{-k}}_\lambda L_{-2k}]}_{\lambda+\mu} L_{3k}]]\in \mathbb{C}^*L_0. \]
		Hence $3k\in I_0$. 	Since $Re(a_k)\geq \frac{3}{2}$ and  $a_{3k}=3a_{k}-3$, one can see that $p_{k,3k}(\partial,\lambda)\neq 0$ and $\text{deg}\  p_{k,3k}(\partial,\lambda)\leq 1$ by Propositions 3.1 and 3.2.  Let us consider the following Jacobi-identity 
		\begin{equation}\label{eess}  [[{L_{k}}_\lambda L_{k}]]_{\lambda+\mu} L_{2k}]=[{L_k}_\lambda [{L_k}_\mu L_{2k}]]-[{L_k}_\mu [{L_k}_\lambda L_{2k}]]. 
		\end{equation}
		Since $p_{k,2k}(\partial,\lambda)=c_{k,2k}$ for some  $c_{k,2k}\in \mathbb{C}^*$,
		one can see that  \[p_{k,k}(-\lambda-\mu,\lambda)p_{2k,2k}(\partial,\lambda+\mu)=c_{k,2k}p_{k,3k}(\partial,\lambda)-c_{k,2k}p_{k,3k}(\partial,\mu).\] 
		Hence $p_{2k,2k}(\partial,\lambda)=0$ and $p_{k,3k}(\partial,\lambda)\in \mathbb{C}^*$. Hence $a_{4k}=4a_k-4$.

		Let us first consider the case that $4k\in I_1$ and $a_{-k}\neq 1$.
		Notice that \[p_{-k,k}(\partial,\lambda),\ p_{-k,2k}(\partial,\lambda),\ p_{k,-2k}(\partial,\lambda)\in \mathbb{C}^*.\] Hence $b_{k}=0$. Otherwise, $L^0(\mathcal{L}[k])$ shall be isomorphic to centerless Virasoro Lie algebra  by \cite[Theorem 1]{Ka}, which is impossible because $p_{2k,-2k}(\partial,\lambda)=0$. 
		Hence, by Proposition {\ref{bcp}}, $\mathcal{L}[4k]$ is either isomorphic to $\mathcal{V}(0)$ or $CL_2(b,0)$ or $SCL_2(b,0)$  for some $b\in \mathbb{C}^*$. Since $\mathcal{L}[4k]$ is a subalgebra of $\mathcal{L}[k]$ and $Re(a_{-4k})\leq Re(a_{4k})$, we can see that $p_{k,jk}(\partial,\lambda)\neq 0$ for each $j\geq 0$.    
		
		Let $\mathcal{J}_i=\bigoplus_{t\in \mathbb{Z}}\mathcal{L}_{i+4tk}$ for $i=1,2,3$. Then for each $i$, $\mathcal{J}_i$ can be regarded as a free intermediate series module of $\mathcal{L}[4k]$ concerning the adjoint action. We shall distinguish the following cases.

		\noindent {\bf Case I}  $\mathcal{L}[4k]\cong \mathcal{V}(0)$.
		In this case,  we have the following table
		\begin{table}[htb]   
			\begin{center}   
				\begin{tabular}{|c|c|c|c|c|c|c|c|c|c|}   
					\hline
					t  &-4  &-3 &-2 &-1 &0 &1 &2 &3 &4 \\ 
					\hline
					$a_{tk}$ &2 & $\frac{3}{2}$&  1  &$\frac{3}{2}$ & 2 & $\frac{3}{2}$ & 1 &  $\frac{3}{2}$ &2  \\ 
					\hline   
				\end{tabular}   
			\end{center}   
		\end{table}

		We  can assume that $p_{4ik,4jk}(\partial,\lambda)=\partial+2\lambda$ for each $i,j\in \mathbb{Z}$. Furthermore, from \cite{WCY}, we may assume that 
		\[ [{L_{4tk}}_\lambda L_{1+4jk}]=(\partial+\frac{3}{2}\lambda)L_{1+4(t+j)k},\]
		\[           [{L_{4tk}}_\lambda L_{3+4jk}]=(\partial+\frac{3}{2}\lambda)L_{3+4(t+j)k}, \]
		for each $t,j\in \mathbb{Z}$.
		In particular, we can see that $a_{1+4tk}=a_{3+4tk}=\frac{3}{2}$ for each $t\in \mathbb{Z}$. Hence $a_{2+4tk}=1$ for each $t\in \mathbb{Z}$.  In addition,
		\[[{L_{4tk}}_\lambda L_{2+4jk}]=(\partial+\lambda)L_{2+4(t+j)k}\]
		for each $t,j\in \mathbb{Z}$.
		Considering the Jacobi-identity of $L_{k},L_{2k},L_{4k}$, we get 
		\[  \frac{1}{2}p_{k,2k}(-\lambda-\mu,\lambda)(\partial+3\lambda+3\mu)=p_{k, 6k}(\partial,\lambda)\mu-\frac{1}{2}p_{2k, 5k}(\partial,\mu)(\partial+3\lambda+\mu).\]
		Hence  \[ p_{k,2k}(\partial,\lambda)=p_{k, 6k}(\partial,\lambda)=-p_{2k, 5k}(\partial,\lambda)\in \mathbb{C}^*.\]
		Besides considering the Jacobi-identity of $L_k,L_{2k},L_{3k}$,  
		we have \[ p_{k,2k}(-\lambda-\mu,\lambda)p_{3k,3k}(\partial,\lambda+\mu)=p_{2k,3k}(\partial+\lambda,\mu)p_{k,5k}(\partial,\lambda )-p_{k,3k}(\partial+\mu,\lambda)\mu.\]
		Since \[p^1_{3k,3k}(\partial,\lambda)=c_{3k,3k}(\partial+2\lambda+2\mu),\]
		and 
		\[p^1_{k,5k}(\partial,\lambda )=c_{k,5k}(\partial+2\lambda),\]
		for some $c_{3k,3k}$, $c_{k,5k}\in \mathbb{C}^*$, 
		one can see that
		\[ 2p_{k,2k}(\partial,\lambda)c_{3k,3k}=-p_{k,3k}(\partial,\lambda).\]
		On the other hand,  from the Jacobi-identity of $L_k, L_{3k}, L_{3k}$, we have
		\[   p_{k,3k}(-\lambda-\mu,\lambda)(\partial+\frac{3}{2}\lambda+\frac{3}{2}\mu)=c_{3k,3k}p_{k,6k}(\partial,\lambda)(\partial+2\mu+\lambda)-\frac{1}{2}p_{k,3k}(\partial+\mu,\lambda)(\partial+3\mu).  \]
		Hence 
		\[  \frac{3}{2}p_{k,3k}(\partial,\lambda)=c_{3k,3k}p_{k,6k}(\partial,\lambda)=c_{3k,3k}p_{k,2k}(\partial,\lambda)=-\frac{1}{2}p_{k,3k}(\partial,\lambda).\]
		This implies that $p_{k,3k}(\partial,\lambda)=0$, which is impossible.

		\noindent {\bf Case II}  $\mathcal{L}[4k]\cong CL_2(b,0)$ for some nonzero $b\in \mathbb{C}$. Hence for each $i,j\in \mathbb{Z}$, we may assume that 
		\[p_{4ik,4jk}(\partial,\lambda)=(4i+b)\partial+(4i+4j+2b)\lambda.\] In addition, $a_{4ik}=\frac{4ik+2b}{b}$ for each $i\in \mathbb{Z}$.

		From \cite[Theorem 6.19]{HPC},  there exists some $N>0$ and for each $i>N$,
		we can assume that  \[  p^1_{4ik,k}(\partial,\lambda)=(4ik+b)\partial+ba_{(4i+1)k}\lambda\]
		\[p^1_{4ik,3k}(\partial,\lambda)=(4ik+b)\partial+ba_{(4i+3)k}\lambda\]
		\[p^1_{4ik,2k}(\partial,\lambda)=(4ik+b)\partial+ba_{(4i+2)k}\lambda.\]
		On the other hand, we have
		\[a_{(4i+3)k}=a_{4ik}+a_{3k}-2=a_{4ik}+\frac{3}{4}a_{4k}-2=\frac{3k+4ik+\frac{3}{2}b}{b},\]
		and
		\[a_{(4i+1)k}=a_{4ik}+a_{k}-2=a_{4ik}+\frac{a_{4k}+4}{4}-2=\frac{k+4ik+\frac{3}{2}b}{b},\]
		and 
		\[a_{(4i+2)k}=a_{4ik}+a_{2k}-2=a_{4ik}+\frac{a_{4k}}{2}-2=\frac{2k+4ik+b}{b}.\]
		By Proposition 3.1(i), we can see that for each $i>N$,  $\text{deg}  p_{k,4ik+k}(\partial,\lambda)=1$ and \[\text{deg}\  p_{k,4ik+2k}(\partial,\lambda)=\text{deg}\   p_{k,4ik+3k}(\partial,\lambda)=0,\      p_{4ik+2k,4ik+3k}(\partial,\lambda)\in \mathbb{C}.\ \ \] 
		Hence the Jacobi-identity of $L_k,L_{(4i+2)k},L_{4(i+1)k}$ leads to 
		\begin{equation}{\label{eq3.11}}
			\begin{aligned}
				&p^1_{k,(4i+2)k}(-\lambda-\mu,\lambda)((4ik+3k+\frac{1}{2}b)\partial+(8ik+7k+\frac{3}{2}b)(\lambda+\mu))\\=&p^1_{k, (8i+6)k}(\partial,\lambda)((4ik+2k)(\partial+\lambda)+(8ik+6k+b)\mu)\\&-p^1_{(4i+2)k, (4i+5)k}(\partial,\mu)((k+\frac{1}{2}b)(\partial+\mu)+(4ik+5k+\frac{3}{2}b)\lambda).
			\end{aligned}
		\end{equation}
		Equation {\eqref{eq3.11}} forces  that \[ p_{k,(4i+2)k}(\partial,\lambda)=p_{k, (8i+6)k}(\partial,\lambda)=-p_{(4i+2)k, (4i+5)k}(\partial,\lambda)\in \mathbb{C}^*.\]
		Besides considering the Jacobi-identity of $L_k,L_{4ik+2k},L_{4ik+3k}$,  
		we have
		\begin{equation}\label{eq3.13}
			\begin{aligned}
				&p^1_{k,4ik+2k}(-\lambda-\mu,\lambda)p^1_{4ik+3k,4ik+3k}(\partial,\lambda+\mu)\\=&p^1_{4ik+2k,4ik+3k}(\partial+\lambda,\mu)p^1_{k,8ik+5k}(\partial,\lambda )\\&-p^1_{k,4ik+3k}(\partial+\mu,\lambda)((4ik+2k)\partial+(8ik+6k+b)\mu)).
			\end{aligned}
		\end{equation}
		Since \[p^1_{4ik+3k,4ik+3k}(\partial,\lambda)=c_{4i+3k,4i+3k}(\partial+2\lambda+2\mu),\]
		for some $c_{4ik+3k,4ik+3k}\in \mathbb{C}^*$,
		The coefficients before $\mu$ of the Equation {\eqref{eq3.13}} forces that 
		one can see that
		\[ 2p^1_{k,4ik+2k}(\partial,\lambda)c_{4ik+3k,4ik+3k}=-(8ik+6k+b)p^1_{k,4ik+3k}(\partial,\lambda)\in \mathbb{C}^*.\]
		On the other hand,  from the Jacobi-identity of $L_k, L_{4ik+3k}, L_{4ik+3k}$, we have
		\begin{equation}{\label{eq3.12}}
			\begin{aligned}   &p^1_{k,4ik+3k}(-\lambda-\mu,\lambda)((4ik+4k+b)\partial+(8ik+7k+\frac{3}{2}b)(\lambda+\mu))\\=&c_{4ik+3k,4ik+3k}p^1_{k,8ik+6k}(\partial,\lambda)(\partial+2\mu+\lambda)\\&-p^1_{k,4ik+3k}(\partial+\mu,\lambda)((3ik+\frac{1}{2}b)\partial+(8ik+7k+\frac{3}{2}b)\mu).  
			\end{aligned}  
		\end{equation}
		The coefficients before $\lambda$ of the Equation {\eqref{eq3.12}} forces that 
		\begin{equation*}
			\begin{aligned}
				&(8ik+7k+\frac{3}{2}b)p^1_{k,4i+3k}(\partial,\lambda)=c_{4ik+3k,4ik+3k}p^1_{k,8ik+6k}(\partial,\lambda)\\=&c_{4ik+3k,4ik+3k}p^1_{k,4ik+2k}(\partial,\lambda)=-\frac{1}{2}(8ik+6k+b)p^1_{k,4ik+3k}(\partial,\lambda).
			\end{aligned}
		\end{equation*}
		Hence $p_{k,4ik+3k}(\partial,\lambda)=0$, which is impossible.

		\noindent {\bf Case III}, $\mathcal{L}[4k]\cong SCL_2(b,0)$. Then there exits some $jk\in I_2$ with $a_{-jk}=1$ and $a_{jk}=4$. In addition, 
		\[p^1_{-jk,jk}(\partial,\lambda)=\lambda(\partial-2\lambda).\]  Since $\mathcal{L}[4k]$ is a subalgebra of $\mathcal{L}[k]$, we have $p_{-k,-ik}(\partial,\lambda)\leq 0$ and $p_{k,ik}(\partial,\lambda)\leq 0$ for each $i\geq 0$. Notice that if $p_{-k,jk}(\partial,\lambda)\neq 0$, then  $\text{deg} p_{-k,jk}(\partial,\lambda)\leq 1$ by Lemma 3.1(2).  Consequently,  we can  also obtain a contradiction from the coefficients before the $\lambda^2$ for 
		the Jacobi-identity of $L_{-k}$,$L_{(-j+1)k}$ and $L_{jk}$.

		Next, Let us deal with the case that $4k\not\in I_1$. Notice that
		\[a_{4k}+a_{-4k}\leq a_{k}+a_{3k}+a_{-k}+a_{-3k}-2\leq 4.\]
		Hence $p_{-4k,4k}(\partial,\lambda)=0$ or $4k\in I_0$. If $4k\in I_0$, then $\text{deg}\  p_{-k,-3k}(\partial,\lambda)=1$. 
		
		By considering the Jacobi-identity of $L_{-k},L_{-k}$ and $L_{-2k}$, we can find that $a_{-4k}=0$. Hence \[ 0=a_{-4k}=a_{-k}+a_{-3k}-2=4a_{-k}-5.\]
		It implies that $a_{-k}=\frac{5}{4}$. Now suppose that $p_{-4k,4k}(\partial,\lambda)=0$. Then the Jacobi-identity of $L_{-k},L_{-3k}$ and $L_{4k}$ implies that 
		\begin{equation}{\label{zh}}c_{-k,k}p^1_{-3k,4k}(\partial+\lambda,\mu)=p^1_{-k,4k}(\partial+\mu,\lambda)c_{-3k,k}
		\end{equation}
		Since $Re(a_{k})\geq \frac{3}{2}$, $Re(a_{4k})\geq 2$. Hence $p^1_{-k,4k}(\partial,\lambda)\neq 0$. Since $p_{k,3k}(\partial,\lambda)\in \mathbb{C}^*$, $\text{deg}\   p^1_{-k,4k}(\partial,\lambda)=1$. Hence Equation \eqref{zh} forces that  $a_{3k}=a_{-k}-1$, which is impossible. In conclusion, if $4k\not\in I_1$,
		then $a_{-k}=\frac{5}{4}$. In  this case, we have
		$a_{-3k}=\frac{3}{4}$. Since $3k\in I_0$ and $p_{-3k,-3k}(\partial,\lambda)\neq 0$, $12k\in I_1$ and $a_{-12k}=4a_{-3k}-4=-1$.   Then investigating the structure of $\mathcal{L}[3k]$ as in {\bf Case I}, {\bf Case II}, {\bf Case III}, we can also get a contradiction.

		As for the case that
		$4k\in I_1$ and $a_{-4k}=1$, one can see that  $a_{-k}=\frac{5}{4}$. Thus it can be dealt  as the case  $4k\not\in I_1$. We can also get a contradiction.  In conlusion, $p_{-k,-k}(\partial,\lambda)=0$.   
	\end{proof}
	\begin{corollary}If $k\in I_0$ and $Re(a_{-k})\leq Re(a_k)$, then $a_{-k}=0$. Moreover, $\mathcal{L}[k]$ is not simple and has a proper ideal which is isomorphic to $CL_1(s)$. In particular, $a_{nk}=n+2$ for each $n\geq 0$. \end{corollary}
	\begin{proof}
		By Propositon 3.6, $p_{-k,-k}(\partial,\lambda)=0$. Then the Jacobbi-identity of $L_{-k}$,$L_{-k}$ and $L_{k}$ forces that $a_k=0$. In addition, one can also check that \[ \mathcal{J}=\oplus_{j\geq 0}\mathbb{C}[\partial]L_{jk}\oplus \mathbb{C}[\partial](\partial+b_{-k})L_{-k}\] is a proper ideal of  $\mathcal{L}[k]$.    By Proposition \ref{l2},  $\mathcal{J}$ is isomorphic to $CL_1(s)$. 
		
	\end{proof}
	
	\begin{p}{\label{3.7}} $Supp(\mathcal{L})$ is either $k\mathbb{Z}$ or $k\mathbb{Z}_{\geq  -1}$ for some integer $k$. \end{p}
	\begin{proof}
		Suppose that $Supp(\mathcal{L})\neq k\mathbb{Z}_{\geq  -1}$ for any $k$.  Let $t=min Supp(\mathcal{L})_1$. We may assume that $Re(a_{-t})\leq Re(a_t)$.  By Lemma \ref{l3.2} and Corollary 3.2,  $a_{jk}=a_{k}+(j-1)(a_{k}-2)$ for each  $j\geq 1$.

		Set $S:=Supp(\mathcal{L})\backslash t\mathbb{Z}$. Suppose that $S\neq \emptyset$ and  set $\mathcal{S}:=\bigoplus_{i\in S}\mathbb{C}[\partial]L_i$. Then one can check directly that 
		\[\mathcal{I}:=\mathcal{S}+ [\mathcal{S}_{(\lambda)}\mathcal{S}]\] is a non-zero ideal of $\mathcal{L}$, where 
		\[[\mathcal{S}_{(\lambda)}\mathcal{S}]:=\{a_{(n)}b|a,b\in \mathcal{S},\ n\in \mathbb{N}\}.\] Since  $\mathcal{L}$ is simple, $L_0\in \mathcal{I}$.  Thus there exists some $j\in S$ such that $j\in Supp_1(\mathcal{L})$.  Let $t_0=min Supp(\mathcal{L})_1\backslash 
		t\mathbb{N}$.  Since  $Re(a_t)>1$,  $p_{-t_0,t}(\partial,\lambda)\neq 0$. It implies that $t-t_0\in Supp(\mathcal{L})$.  On the other hand, by the choice of $t,t_0$,  we have $t_0-t\not\in Supp(\mathcal{L})_1$.  Let us consider the  Jacobi-identity of $L_{t_0},L_{-t}$ and $L_{t-t_0}$. We obtain that
		\begin{equation}{\label{3223}} p_{-t,t-t_0}(\partial+\lambda,\mu)p_{t_0,-t_0}(\partial,\lambda)=p_{t_0,t-t_0}(\partial+\mu,\lambda)p_{-t,t}(\partial,\mu).\end{equation}
		By Proposition {\ref{p3.1}}, we may analyze Equation (\ref{3223}) into the following three cases.

		\noindent {\bf Case I}:\ $p_{-t,t-t_0}(\partial,\mu)\neq 0$ and  $Re(a_{-t_0}) \leq Re(a_{t_0})$. Then Equation {\eqref{3223}} forces that either   $a_{-t}=1$ or $t\in I_0$. Thus $a_{jt}=2+j$ for each $j\geq 0$. On the other hand, since $Re(a_{t_0})>1$, Equation {\eqref{3223}} forces that $p_{t_0,-t_0}(\partial,\lambda)\in \mathbb{C}$. Thus $a_{jt_0}=j+2$ for each $j\geq 0$. It implies that $2+t_0=a_{t_0t}=2+t$, which is impossible.

		\noindent {\bf Case II}:\ $p_{-t,t-t_0}(\partial,\mu)\neq 0$ and   $Re(a_{-t_0}) \geq Re(a_{t_0})$. On the one hand, Equation {\eqref{3223}} forces that either   $a_{-t}=1$ or $t\in I_0$. In both cases, $a_{t_0t}=2+t_0>2$. On the other hand,
		Equation {\eqref{3223}} forces that either $a_{t_0}=1$ or $t_0\in I$. Since $Re(a_{-t_0}) \geq Re(a_{t_0})$, we can see that $a_{-t_0}\geq 3$. Hence $p_{-t_0,t_0t}(\partial,\lambda)\neq 0$ by Lemma 3.1(2) and 
		$\text{deg}\  p_{-t_0,t_0t}(\partial,\lambda)\leq 1$.  It implies that $a_{(t-1)t_0}>a_{tt_0}\geq 3$. Inductively, one can see that $a_{(t-i)t_0}>a_{(t+1-i)t_0}\geq 3$ for each $i<t$. Finally, one can see that $a_{t_0}>a_{tt_0}$, which is impossible.

		\noindent {\bf Case III}: If one of $p_{-t,t-t_0}(\partial,\lambda)$ and $p_{t_0,t-t_0}(\partial,\lambda)$ is zero, then $p_{-t,t-t_0}(\partial,\lambda)=p_{t_0,t-t_0}(\partial,\lambda)
		=0$.  Considering the  Jacobi-identity of $L_{t},L_{-t},L_{t_0}$,  we have
		\begin{equation}{\label{zzh}}p_{t,-t}(-\lambda-\mu,\lambda)(\partial+a_{t_0}\lambda+a_{t_0}\mu+b_{t_0})=p_{t,t_0-t}(\partial,\lambda)p_{-t,t_0}(\partial+\lambda,\mu).
		\end{equation}
		Thus $(\partial+a_{t_0}\lambda+a_{t_0}\mu+b_{t_0})|p_{t,t_0-t}(\partial,\lambda)$ or $(\partial+a_{t_0}\lambda+a_{t_0}\mu+b_{t_0})|p_{-t,t_0}(\partial+\lambda,\mu)$. The former case implies that $a_{t_0}=0$ while the latter case implies that $a_{t_0}=1$.  It implies that $a_{-t_0}\geq 3$. By using a similar proof as in {\bf Case II}, we can also get a contradiction. 
		
	\end{proof}
	
	\large{\emph{Proof of Theorem \ref{mt}.}}
	\begin{proof} By Proposition \ref{3.7},  we may assume that $1\in Supp(\mathcal{L})_1$.    By Proposition {\ref{l2}} and {\ref{bcp}},   we only need to consider the case that $\mathcal{L}\neq \mathcal{L}[1]$, that is $Supp(\mathcal{L})=\mathbb{Z}$ and
		$p_{-1,-1}(\partial,\lambda)=0$. Hence by Proposition 3.4 and Corollary 3.2, one can see that $a_{i}=i+2$ for each $i\geq 0$. 
		
		Next we shall prove that $i\in Supp_1(\mathcal{L})$ for each $i>0$. Suppose that $p_{-2,2}(\partial,\lambda)=0$. Then the  
		Jacobi-identity of $L_1$,$L_1$ and $L_{-2}$ implies that $p_{1,-2}(\partial,\lambda)\in \mathbb{C}$ and $1\in I_0$. Notice that $p_{1,-2}(\partial,\lambda)\neq 0$. Otherwise, $\bigoplus_{i\leq -2}\mathcal{L}_{i}$ will be a proper ideal of $\mathcal{L}$.  From the Jacobi-identity of 
		$L_1$, $L_{-2}$ and $L_{2}$, we have
		\[p_{1,-2}(-\mu-\lambda,\mu)p_{-1,2}(\partial,\lambda+\mu)=p_{1,2}(\partial+\mu,\lambda)p_{-2,3}(\partial,\lambda).\]
		Since $p_{1,-2}(\partial,\lambda)$, $p_{-1,2}(\partial,\lambda)\in \mathbb{C}^*$, it forces that $p_{1,2}(\partial,\lambda)\in \mathbb{C}^*$, which is a contradition. Hence $2\in Supp_{1}(\mathcal{L})$. 
		Suppose that there exists some $i>1$ such that $i\in Supp_1(\mathcal{L})$ but $1+i\not\in Supp_1(\mathcal{L})$. 			Let us consider the  Jacobi-identity of $L_{-1},L_{-i}$ and $L_{1+i}$. We obtain that
		\begin{equation}{\label{efnal}} p_{-i,1+i}(\partial+\lambda,\mu)p_{-1,1}(\partial,\lambda)=p_{-1,1+i}(\partial+\mu,\lambda)p_{-i,i}(\partial,\mu).\end{equation}
		Since $i>1$ and $a_i>3$, one can see that $i\in I_1\cup I_2$. Besides, Equation {\ref{efnal}} forces that $a_{-i}=1$ and $i\in I_1$. It implies that $a_i=4-a_{-i}=3$, which is a contradiction. Hence  $i\in Supp_1(\mathcal{L})$ for each $i>0$.
		
		Finally, we distinguish the following two cases.

		{\bf Case I}. $1\in I_1$. For each $i>2$, since $a_i>4$, we have $i\in I_1$. If
		$2\in I_1$, then $a_i=2+i$ for each $i\in \mathbb{Z}$.  Thus $\mathcal{L}$ is quadratic and is isomophic to $CL_2(1,b)$ for some $b\in \mathbb{C}$. It contradicts the simplicity of $\mathcal{L}$. Suppose that	$2\in I_2$. Then $a_{-2}=1$ and $a_i=2+i$ for $i\neq -2$.
		By using a similar proof as in (2) of Proposition \ref{l2}, one can see that $\mathcal{L}\cong SCL_2(1,b)$ for some  $b \in \mathbb{C}$.

		{\bf Case II}. $1 \in I_0$. Notice that $i\in I_1\cup I_2$ for each $i>0$.
		Then
		\[\text{deg}\  p_{-1-i,i}(\partial,\lambda)=a_{-1-i}+a_{i}-a_{-1}-1\geq 1\]
		for each $i>0$. Since $a_i>1$ for each $i>0$,  by (3.3), 
		\[ \mathcal{J}:=\oplus_{j\neq -1}\mathbb{C}[\partial]L_{j}\oplus \mathbb{C}[\partial](\partial+b_{-1})L_{-1}\] 
		is a non-zero proper ideal of $\mathcal{L}$, which is impossible.  
	\end{proof}

\end{document}